\documentclass[11pt]{article}
\usepackage[T1]{fontenc}
\usepackage[utf8]{inputenc}
\usepackage[margin=1.08in]{geometry}
\usepackage{amsmath,amssymb,amsthm,mathtools}
\usepackage{enumitem}
\usepackage[colorlinks=true,linkcolor=blue,citecolor=blue,urlcolor=blue]{hyperref}

\newtheorem{theorem}{Theorem}[section]
\newtheorem{proposition}[theorem]{Proposition}
\newtheorem{lemma}[theorem]{Lemma}
\newtheorem{corollary}[theorem]{Corollary}
\theoremstyle{remark}
\newtheorem{remark}[theorem]{Remark}

\newcommand{\Z}{\mathbb Z}
\newcommand{\ex}{\operatorname{ex}}

\begin{document}
\title{On zero-sum Ramsey numbers of cycles and wheels}
\author{Cheng Chi\footnote{School of Mathematical Sciences, Shanghai Jiao Tong University, 800 Dongchuan Road, Shanghai 200240, China.
        Email: chengchi@sjtu.edu.cn.
        Supported by National Key R\&D Program of China under grant No. 2022YFA1006400 and National Natural Science Foundation of China No. 12571376.}
    \and Jialin He\footnote{School of Mathematical Sciences,  Key Laboratory of MEA (Ministry of Education) and Shanghai Key Laboratory of PMMP,  East China Normal University, Shanghai 200241, China.
        Email: jlhe@math.ecnu.edu.cn.
        Supported in part by Science and Technology Commission of Shanghai Municipality No. 22DZ2229014.}}
\date{}

\maketitle

\begin{abstract}
    For an integer $q\ge 2$ and a graph $F$ with $q\mid e(F)$, let $R(F,\Z_q)$ be the least integer $n$ such that every edge-labeling
    $w\colon E(K_n)\to \Z_q$ contains a copy of $F$ whose edge-label sum is zero in $\Z_q$.
    Let $C_{qk}$ denote the cycle on $qk$ vertices.
    We prove that
    \[
        R(C_{qk},\Z_q)\le \max\{R(C_{2q},\Z_q),qk+q-1\}
    \]
    via an insertion argument rooted in the classic Erd\H{o}s-Ginzburg-Ziv theorem.
    Combined with Pikhurko's result and Katz, Lian, Malekshahian, and Shapiro's result, we obtain that for every $q\ge 3$,
    $R(C_{2q},\Z_q)\le \min\{35q^2,D_0q\}$, and hence
    $R(C_{qk},\Z_q)\le \max\{\min\{35q^2,D_0q\},qk+q-1\}$, where $D_0$ is a constant independent of $k$ and $q$.
    We also show that $R(C_{qk},\Z_q)\ge qk+q-1$ for odd $q\ge 3$.
    Hence, for every fixed odd $q\ge 3$ and every $k\ge \min\{35q,D_0\}$, we obtain the exact value
    \[
        R(C_{qk},\Z_q)=qk+q-1.
    \]
    For even $q\ge 4$, we show that
    \[qk+\frac q2-1\le R(C_{qk},\Z_q)\le \max\{\min\{35q^2,D_0q\},qk+q-1\},
    \]
    leaving an additive gap of order $q/2$ when $k$ is large.

    Let \(W_m(s)\coloneqq C_m + \overline{K_s}\) be the graph consisting of a cycle of length $m$ and $s$ independent vertices, each of which is adjacent to every vertex on the cycle.
    Note that when $s=1$, this is the wheel graph with a rim of length $m,$ which we denote simply as $W_m.$
    Let $p$ be an odd prime, and let $s\ge1$ be a fixed integer satisfying that $p\mid (s+2).$
    Then, there exists a constant $K_0=K_0(p,s)$ such that for every $k\ge K_0,$
    \[R(W_{pk}(s), \Z_p)=pk+s.\]

    Moreover, for the case $q=3$, we prove that for all \(k \ge 2\), \(R(C_{3k}, \mathbb{Z}_3) = 3k + 2\) and \(R(W_{3k}, \mathbb{Z}_3) = 3k + 1\).
\end{abstract}

\section{Introduction}

Zero-sum Ramsey theory asks for copies of a prescribed graph whose edge labels sum to the identity element of a finite Abelian group.
The subject was initiated by Bialostocki and Dierker~\cite{BD1990,BD1992} and is closely tied to additive zero-sum results such as the Erd\H{o}s-Ginzburg-Ziv theorem~\cite{EGZ1961} and the generalized Cauchy-Davenport theorem~\cite{Cauchy1813,Davenport1935}.

Unless otherwise stated, let $q\ge2$ be an integer.
We work throughout with the cyclic group $\Z_q$.
For a graph $G$, the \emph{zero-sum Ramsey number} $R(G,\Z_q)$ is the least integer $n$ such that every edge-labeling $w\colon E(K_n)\to \Z_q$ contains a copy $G'$ of $G$ with $\sum_{e\in E(G')} w(e)=0$ in $\Z_q$.
Note that the divisibility condition $q\mid e(G)$ is necessary for any general zero-sum Ramsey statement.
In fact, if $q\nmid e(G)$, then for every $n$, in the edge-labeling of $K_n$ defined by $w(e)=1$ for every $e$, every copy $G'$ of $G$ satisfies $\sum_{e\in E(G')} w(e)=1\cdot e(G)\ne 0$ in $\Z_q$.
Whenever $q\mid e(G)$, the existence of the zero-sum Ramsey number is guaranteed by the standard multicolor Ramsey theorem.
Indeed, any monochromatic copy of $G$ arising from a $q$-coloring is necessarily zero-sum.

General upper bounds have been established for certain families of graphs, such as complete graphs~\cite{AC1993,Caro1992Complete,Ca1994}, stars~\cite{Ca1992}, forests~\cite{CD2026}, bounded degree graphs~\cite{KLMS2025} and $d$-degenerate graphs~\cite{Shapiro2026}.
However, the exact value of \(R(G, \Z_q)\) is known for only a few special instances of $G$ and $q$, such as when \(G\) is a star and \(q\) is even, or when \(G\) is a matching~\cite{BD1992}.
Other notable results on zero-sum Ramsey numbers include the complete determination of $R(G, \Z_2)$~\cite{Caro1994} and the determination of $R(F, \Z_3)$ for forests~\cite{CM2025, ACP2025}.
For further background and early developments in this area, we refer the reader to the survey by Caro~\cite{Caro1996} and the seminal work of Alon and Caro~\cite{AC1993}.

In this paper, we study zero-sum cycles with lengths prescribed as multiples of the modulus.
Cycles are a natural test case; determining the exact value of $R(C_{qk},\Z_q)$ depends on the complex interaction between the cycle length, the modulus, and the underlying zero-sum structure.
Recent work on zero-sum cycles in related settings includes~\cite{CGHS2025,CKMS2025}.

Throughout this paper, all congruences and edge-label sums are taken over $\Z_q$.
For a given graph $H$ and an edge-labeling $w\colon E(K_n)\to \Z_q$, we define the weight of
$H$ as
\[
    w(H)=\sum_{e\in E(H)}w(e)\in \Z_q.
\]
A subgraph $H$ is called \emph{zero-sum} if $w(H)=0$.

For integers $q\ge 2$ and $k\ge 2$, let $C_{qk}$ denote the cycle on $qk$ vertices.
The primary object of this paper is the zero-sum Ramsey number:
\[
    R(C_{qk},\Z_q)=\min\Bigl\{n\colon \text{ every }w\colon E(K_n)\to \Z_q
    \text{ contains a zero-sum }C_{qk}\Bigr\}.
\]

Our first main result reveals that $R(C_{qk},\Z_q)$ is partly controlled by $R(C_{2q},\Z_q)$.

\begin{theorem}\label{thm:reduction}
    For every integer $q\ge 2$ and every $k\ge 2$\footnote{When $k=1$, Caro~\cite{Caro1996} conjectured that $R(C_q,\Z_q)=4q-3$ for odd $q\ge5$.},
    \[
        R(C_{qk},\Z_q)\le \max\{R(C_{2q},\Z_q),qk+q-1\}.
    \]
    Consequently, if $qk+q-1\ge R(C_{2q},\Z_q)$, then every labeling of $K_{qk+q-1}$ by elements of $\Z_q$ contains a zero-sum $C_{qk}$.
\end{theorem}

This reduction becomes particularly effective when paired with an upper bound of $R(C_{2q},\Z_q)$.
By applying Pikhurko's Tur\'an bound for even cycles~\cite{Pikhurko2012}, and Katz, Lian, Malekshahian, and Shapiro's zero-sum Ramsey result for bounded degree graphs~\cite{KLMS2025}, we obtain the following estimation for $R(C_{2q},\Z_q)$.

\begin{theorem}\label{thm:basecase}
    For every integer $q\ge 3$, there exists a positive integer $D_0$ independent of $q$ such that
    \[
        R(C_{2q},\Z_q)\le \min\{35q^2,D_0q\}.
    \]
    Consequently, for every $q\ge 3$ and $k\ge 2$,
    \[
        R(C_{qk},\Z_q)\le \max\{\min\{35q^2,D_0q\},qk+q-1\}.
    \]
\end{theorem}

To establish lower bounds for $R(C_{qk},\Z_q)$, we prove the following theorem, which depends on the parity of the modulus.

\begin{theorem}\label{thm:lower}
    Let $q\ge 2$ and $k\ge 1$.
    \begin{enumerate}[label=\textup{(\roman*)}]
        \item If $q$ is odd, then $R(C_{qk},\Z_q)\ge qk+q-1$.
        \item If $q$ is even and $q\ge 4$, then $R(C_{qk},\Z_q)\ge qk+\frac q2-1$.
    \end{enumerate}
\end{theorem}

As a result, the bound for the case when $q$ is odd is exact for all sufficiently large \(k\).
For the case when $q$ is even, our arguments leave a small additive gap of order \(q/2\).

\begin{corollary}\label{cor:odd-exact}
    Let $q\ge 3$ and $k\ge \min\{35q,D_0\}$, where $D_0$ is the positive integer from Theorem~\ref{thm:basecase}.
    \begin{enumerate}[label=\textup{(\roman*)}]
        \item If $q$ is odd, then $R(C_{qk},\Z_q)= qk+q-1$.
        \item If $q$ is even and $q\ge 4$, then $qk+\frac q2-1
                  \le R(C_{qk},\Z_q)
                  \le qk+q-1.$
    \end{enumerate}
\end{corollary}

\begin{remark}
    The statements above do not require $q$ to be prime or a prime power.
\end{remark}

For an integer $s\ge1$, let $\overline K_s$ denote the empty graph on $s$ vertices, and let $W_m(s)=C_m+\overline K_s$ denote the join of $C_m$ and $\overline K_s$, where the $s$ vertices of $\overline K_s$ are mutually non-adjacent and each of them is adjacent to every vertex of the cycle.
Note that the wheel graph $W_m$ corresponds to the special case when $s=1$.

\begin{theorem}\label{thm:prime-case}
    Let $p$ be an odd prime and let $s\ge1$ be a fixed integer satisfying $s+2\equiv0\pmod p$.  Let $F_0=4(p-2)$ and
    \begin{equation*}
        K_0(p,s)=\max\left\{4,\left\lceil\frac{(s+2F_0)p^{F_0}+F_0-s}{p}\right\rceil\right\}.
    \end{equation*}
    Then, for every $k\ge K_0(p,s)$,
    \[
        R(W_{pk}(s),\Z_p)=pk+s.
    \]
\end{theorem}

For the specific modulus \(q = 3\), we can establish a sharper exact result that holds without the large-\(k\) assumption.

\begin{theorem}\label{thm:q3-cycles}
    For every $k\ge 2$, $R(C_{3k},\Z_3)=3k+2$.
\end{theorem}

Extending our techniques, we also prove a corresponding result for wheel graphs modulo $3$.

\begin{theorem}\label{thm:q3-wheel}
    For every $k\ge 2$, $R(W_{3k},\Z_3)=3k+1$.
\end{theorem}

We briefly outline our proof strategy.
Theorem~\ref{thm:reduction} follows from an insertion lemma: given a zero-sum cycle of length at least $2q-1$ and at least $2q-1$ unused vertices, the classic Erd\H{o}s-Ginzburg-Ziv theorem guarantees that we can insert $q$ vertices such that the total change to the cycle's weight is zero.
We prove Theorem~\ref{thm:basecase} by applying Pikhurko's even-cycle Tur\'an bound to the largest color class and invoking the theorem of Katz et al. on zero-sum Ramsey numbers for bounded degree graphs.
The lower bounds in Theorem~\ref{thm:lower} arise from a specific cut labeling in which every long cycle crosses the cut a positive even number of times.
The proof of Theorem~\ref{thm:prime-case} relies on a different proof strategy.
By the Cauchy-Davenport theorem, having many independent non-zero switches would yield a zero-sum copy; thus, in any counterexample, all switch defects must be confined to a bounded number of vertices.
Then, the congruence condition $s+2\equiv0\pmod p$ guarantees that any sufficiently large set of the same type contains a zero-sum copy.
Finally, the exact results for cycles and wheels modulo $3$ require a more structural analysis of labelings that avoid zero-sum copies.
\medskip

The rest of the paper is organized as follows.
Section~2 contains the proofs of Theorem~\ref{thm:reduction} and Theorem~\ref{thm:basecase}.
The proof of Theorem~\ref{thm:lower} is presented in Section~3.
In Section~4, we complete the proof of Theorem~\ref{thm:prime-case}.
Sections~5 and~6 establish the exact results for cycles and wheels modulo 3, proving Theorems~\ref{thm:q3-cycles} and~\ref{thm:q3-wheel}, respectively.

\section{Preliminaries and Proofs of Theorem~\ref{thm:reduction} and Theorem~\ref{thm:basecase}}

In this section, we first introduce a key lemma, which we apply to prove Theorem~\ref{thm:reduction} and Theorem~\ref{thm:basecase}.
Our approach relies on the following version of the Erd\H{o}s--Ginzburg--Ziv theorem~\cite{EGZ1961}.

\begin{theorem}[Erd\H{o}s--Ginzburg--Ziv~{\cite{EGZ1961}}]\label{thm:egz}
    Every sequence of $2q-1$ elements of $\Z_q$ contains a subsequence of length $q$ whose sum is $0$ in $\Z_q$.
\end{theorem}

\begin{lemma}\label{lem:insert}
    Let $q\ge 2$.  Suppose that $K_n$ has an edge-labeling
    \(w\colon E(K_n)\to \Z_q\),
    which contains a zero-sum cycle $C_\ell$ with $w(C_\ell)=\sum_{e\in E(C_\ell)}w(e)\equiv 0\pmod q$ for some  $\ell\ge 2q-1$.
    If there are at least $2q-1$ vertices outside $C_\ell$, then $K_n$ contains a zero-sum cycle $C_{\ell+q}$.
\end{lemma}

\begin{proof}
    Choose $2q-1$ distinct edges of $C_\ell$, say
    $e_i=a_i b_i$, for $i\in[2q-1]$,
    and choose $2q-1$ distinct vertices outside $C_\ell$, say $u_1,\dots,u_{2q-1}$.
    For each $i$, define the insertion increment
    \[
        d_i=w(a_i u_i)+w(u_i b_i)-w(a_i b_i)\in \Z_q.
    \]
    By Theorem~\ref{thm:egz}, there is an index set
    $I\subseteq [2q-1]$ with $|I|=q$,
    such that $\sum_{i\in I}d_i=0$ in $\Z_q$.

    For each $i\in I$, replace the edge $a_i b_i$ of $C_\ell$ by the two-edge path $a_iu_ib_i$.
    Subdividing distinct edges of a simple cycle by distinct new vertices produces another simple cycle.  Thus, we get a cycle of length $\ell+|I|=\ell+q$.
    Moreover, the edge-label sum is $w(C_\ell)+\sum_{i\in I}d_i=0+0=0$ in $\Z_q$.
    Hence the resulting cycle is a zero-sum $C_{\ell+q}$.
\end{proof}

Now, we are ready to prove Theorem~\ref{thm:reduction} and Theorem~\ref{thm:basecase}.

\begin{proof}[Proof of Theorem~\ref{thm:reduction}]
    We first claim that for every $t\ge 3$,
    \[
        R(C_{qt},\Z_q)\le \max\{R(C_{q(t-1)},\Z_q),qt+q-1\}.
    \]
    Indeed, let $n\ge \max\{R(C_{q(t-1)},\Z_q),qt+q-1\}$
    and label the edges of $K_n$ arbitrarily by elements of $\Z_q$.
    Since $n\ge R(C_{q(t-1)},\Z_q)$, there is a zero-sum cycle of length $q(t-1)$.
    Since $n\ge qt+q-1,$
    the number of vertices outside this cycle is at least
    \[
        n-q(t-1)
        \ge (qt+q-1)-q(t-1)
        =2q-1.
    \]
    Because $q(t-1)\ge 2q-1$ for every $t\ge 3$, by Lemma~\ref{lem:insert}, we obtain a zero-sum cycle of length $q(t-1)+q=qt$.
    Thus, we have $R(C_{qt},\Z_q)\le \max\{R(C_{q(t-1)},\Z_q),qt+q-1\}$ for every $t\ge 3$.

    Iterating this process from $t=3$ up to $t=k$ yields
    \[
        R(C_{qk},\Z_q)\le \max\{R(C_{2q},\Z_q),4q-1,\dots,qk+q-1\}
        =\max\{R(C_{2q},\Z_q),qk+q-1\},
    \]
    which completes the proof of Theorem~\ref{thm:reduction}.
\end{proof}

We denote by $\ex(n,F)$ the \emph{Tur\'an number}, given by
\[
    \ex(n,F)=\max\{|E(G)|\colon |V(G)|=n,\,G\text{ contains no copies of }F\}.
\]
The key ingredients for our proof of Theorem~\ref{thm:basecase} are the following two theorems.

\begin{theorem}[Pikhurko~{\cite{Pikhurko2012}}]\label{thm:pikhurko}
    For every integer $r\ge 2$ and every integer $n\ge 1$,
    \[
        \ex(n,C_{2r})\le (r-1)n^{1+1/r}+16(r-1)n.
    \]
\end{theorem}

\begin{theorem}[Katz, Lian, Malekshahian, and Shapiro~\cite{KLMS2025}]\label{thm:KLMS}
    Let $\Delta$ be a positive integer.
    Then there exists a constant $C = C(\Delta)$ such that for any graph $G$ with maximum degree $\Delta$ and any finite Abelian group $\Gamma_0$ such that $|\Gamma_0|$ divides $e(G)$, we have
    \[R(G,\Gamma_0) \le C \cdot v(G).\]
\end{theorem}

\begin{proof}[Proof of Theorem~\ref{thm:basecase}]
    It suffices to prove the two bounds
    \[
        R(C_{2q},\Z_q)\le D_0q
        \qquad\text{and}\qquad
        R(C_{2q},\Z_q)\le 35q^2
    \]
    for every $q\ge 3$.

    To see that $R(C_{2q},\Z_q)\le D_0q,$ let $G=C_{2q}$ and $\Gamma_0=\Z_q.$
    Then, by Theorem~\ref{thm:KLMS}, there exists a constant $C = C(2)$ such that $R(C_{2q},\Z_q)\le 2Cq.$
    Setting $D_0=\lceil 2C\rceil,$ we obtain $R(C_{2q},\Z_q)\le D_0q$ for every $q\ge 3.$
    It remains to prove that $R(C_{2q},\Z_q)\le 35q^2$ for every $q\ge 3.$

    Let $q\ge 3$ and $N=35q^2$.
    Consider an arbitrary labeling $w\colon E(K_N)\to \Z_q$.
    Regard the labels as $q$ colors.
    By the Pigeonhole Principle, there exists a color containing at least $\frac1q\binom N2$ edges.
    If this color contains a $C_{2q}$, then that cycle is zero-sum, since all its edges have the same label $a\in\Z_q$ and
    $2q\cdot a=0$ in $\Z_q$.
    Thus it is enough to prove that $\ex(N,C_{2q})<
        \binom N2/q$.
    By Theorem~\ref{thm:pikhurko}, it suffices to show that
    \[
        (q-1)N^{1+1/q}+16(q-1)N<
        \frac1q\binom N2.
    \]
    Dividing both sides by $(q-1)N$, this becomes
    \[
        N^{1/q}<\frac{N-1}{2q(q-1)}-16.
    \]
    Substituting $N=35q^2$, the right-hand side equals
    \[
        F_q\coloneqq\frac{35q^2-1}{2q(q-1)}-16
        =\frac32+\frac{35}{2(q-1)}-\frac{1}{2q(q-1)}.
    \]
    One can verify that $(35q^2)^{1/q}<F_q$ holds for all $q\ge 3$ (see Appendix~A for details).
    Hence $R(C_{2q},\Z_q)\le 35q^2$ for every $q\ge3$.
    Combining this with $R(C_{2q},\Z_q)\le D_0q$ gives
    $R(C_{2q},\Z_q)\le \min\{35q^2,D_0q\}$.
    The stated bound for $R(C_{qk},\Z_q)$ now follows from Theorem~\ref{thm:reduction}.
    This completes the proof of Theorem~\ref{thm:basecase}.
\end{proof}

\begin{remark}
    In fact, by using Olson's generalization~\cite{Olson1976} of Theorem~\ref{thm:egz} to arbitrary finite Abelian groups, an analogous argument shows that Theorems~\ref{thm:reduction} and~\ref{thm:basecase} hold for any finite Abelian group of order $q.$
    For wheel graphs, a result analogous to Theorem~\ref{thm:reduction} can be derived, and an upper bound of the base case $R(W_{2q},\Z_q)$ can also be established using the results in~\cite{Shapiro2026}.
\end{remark}

\section{Lower Bounds: Proof of Theorem~\ref{thm:lower}}

\begin{proposition}\label{prop:divisor-lower}
    Let $q\ge 2$, $k\ge 1$, and let $d\mid q$.
    Define
    \[
        h=
        \begin{cases}
            d,   & d\text{ is odd},  \\
            d/2, & d\text{ is even}.
        \end{cases}
    \]
    If $h\ge 2$, then
    \[
        R(C_{qk},\Z_q)\ge qk+h-1.
    \]
\end{proposition}

\begin{proof}
    Let $m=qk$.
    We construct a $\Z_q$-labeling of the edges of $K_{m+h-2}$ that contains no zero-sum $C_m$.
    Partition the vertex set as $V=A\cup B$ with $|A|=h-1$ and $|B|=m-1$.
    Assign the edges weights $w(e)$ as follows:
    \[
        w(e)=
        \begin{cases}
            q/d, & \text{if } e\text{ crosses the cut }(A,B),              \\
            0,   & \text{if } e\text{ lies entirely within }A\text{ or }B.
        \end{cases}
    \]
    Note that the element $q/d\in\Z_q$ has additive order $d$.

    Let $C$ be an arbitrary $m$-cycle.
    Since $|B|=m-1<m$, the cycle must use at least one vertex of $A$.
    Similarly, since $|A|=h-1<m$, it also uses at least one vertex of $B$.
    Therefore, $C$ must cross the cut $(A,B)$ a positive even number of times, say $2s$.
    Hence, the total weight of $C$ is
    \[
        w(C)=2s\cdot \frac qd\quad\text{in }\Z_q.
    \]

    Suppose that $C$ uses $t$ vertices of $A$, then each connected component of $C\cap A$ contributes two crossings, and hence $1\le s\le t\le |A|=h-1$.
    If $C$ is a zero-sum cycle in $\Z_q$, then $d\mid 2s$.
    We now consider the parity of $d$.
    If $d$ is odd, then $d\mid 2s$ implies $d\mid s$.
    However, we know that $1\le s\le h-1=d-1$, which yields a contradiction.
    If $d$ is even, by the assumption that $d=2h$,
    $d\mid 2s$ is equivalent to $h\mid s$, which again is a contradiction to $1\le s\le h-1$.
    Thus, the graph contains no zero-sum $m$-cycle, which implies that $R(C_{qk},\Z_q)\ge m+h-1=qk+h-1$.
\end{proof}

Now, we complete the proof of Theorem~\ref{thm:lower}.

\begin{proof}[Proof of Theorem~\ref{thm:lower}]
    If $q$ is odd, we apply Proposition~\ref{prop:divisor-lower} with $d=q$.
    Then, $h=q$, and thus
    \[
        R(C_{qk},\Z_q)\ge qk+q-1.
    \]
    If $q$ is even and $q\ge 4$, we apply Proposition~\ref{prop:divisor-lower} again with $d=q$.
    Then, $h=\frac{q}{2}$, and thus
    \[
        R(C_{qk},\Z_q)\ge qk+\frac q2-1,
    \]
    completing the proof of the lower bounds.
\end{proof}

\section{Proof of Theorem~\ref{thm:prime-case}}
In this section, we prove Theorem~\ref{thm:prime-case}.  We assume throughout that all labels are taken from the field $\Z_p$, where $p$ is an odd prime, and that $s\ge1$ is a fixed integer satisfying the condition
\begin{equation}\label{eq:congruence-condition}
    s+2\equiv0\pmod p.
\end{equation}

For $k\ge1$, recall that $W_{pk}(s)=C_{pk}+\overline K_s$.
The $pk$ vertices of the cycle are called \emph{rim vertices}, while the $s$ vertices of $\overline K_s$ are called \emph{cone vertices}.
In $W_{pk}(s)$, every rim vertex has degree $s+2$, every cone vertex has degree $pk$, and the total number of edges is $|E(W_{pk}(s))|=(s+1)pk$.

As usual, we will use the sum-set notation:
\[A_1+\cdots+A_n=\{a_1+\cdots+a_n\ |\ (a_1,...,a_n)\in A_1\times\cdots\times A_n\}.\]

We shall use the following standard form of the Cauchy-Davenport theorem.

\begin{theorem}[Cauchy-Davenport~{\cite{Cauchy1813,Davenport1935}}]\label{thm:cauchy-davenport}
    If $A,B$ are nonempty subsets of $\Z_p$, then $|A+B|\ge \min\{p,|A|+|B|-1\}$.
\end{theorem}

A repeated application of Theorem~\ref{thm:cauchy-davenport} gives the following consequence.
If $A_1, A_2, \cdots, A_n$ are nonempty
subsets of $\Z_p$, then
$|A_1 + A_2 +\cdots+ A_n| \ge \min(p, |A_1| + |A_2| +\cdots+|A_n|-n+1)$.
Moreover, if $A_1, A_2, \cdots, A_{p-1}$ are
subsets of $\Z_p$ of size $2$, then $A_1 + A_2 +\cdots+ A_{p-1} = \Z_p$.

We also need the following useful lemma.

\begin{lemma}\label{lem:transfer weight}
    Let $U$ be a set with $|U|\ge4$, and let $w \colon E(K_U)\to\Z_p$ be an edge-labeling such that for all distinct $a,b,c,d\in U$,
    \begin{equation}\label{eq:four-point-relation}
        w(ac)+w(bd)=w(ab)+w(cd).
    \end{equation}
    Then, there exist $\lambda_u\in\Z_p$ for each $u\in U$, and a constant $c\in\Z_p$ such that $w(uv)=\lambda_u+\lambda_v+c$ for all distinct $u,v\in U$.
\end{lemma}

\begin{proof}
    Fix two distinct vertices $a,b\in U$.  For distinct $x,y\in U\setminus\{a,b\}$, applying~\eqref{eq:four-point-relation} to the sequences $(x,a,y,b)$ and $(x,b,y,a)$ yields
    \[
        w(xy)+w(ab)=w(xa)+w(yb)
        \quad\text{and}\quad
        w(xy)+w(ab)=w(xb)+w(ya).
    \]
    Hence, $w(xa)-w(xb)=w(ya)-w(yb)$.
    Thus, there is a constant $\delta\in\Z_p$ such that $w(xa)-w(xb)=\delta$ for every $x\in U\setminus\{a,b\}$.
    Set $\lambda_x=w(xa)$ for $x\in U\setminus\{a,b\}$,
    $c=-\delta-w(ab)$, $\lambda_a=\delta+w(ab)$, and $\lambda_b=w(ab)$.
    These definitions immediately yield $w(xa)=\lambda_x+\lambda_a+c$, $w(xb)=\lambda_x+\lambda_b+c$, and $w(ab)=\lambda_a+\lambda_b+c$.  Finally, the first displayed equality implies that  $w(xy)=\lambda_x+\lambda_y+c$ for distinct $x,y\in U\setminus\{a,b\}$.  This completes the proof.
\end{proof}

We now prove Theorem~\ref{thm:prime-case}.

\begin{proof}[Proof of Theorem~\ref{thm:prime-case}]
    Recall that $F_0=4(p-2)$ and $ K_0(p,s)=\max\left\{4,\left\lceil\frac{(s+2F_0)p^{F_0}+F_0-s}{p}\right\rceil\right\}.$
    The lower bound holds trivially, since $W_{pk}(s)$ has $pk+s$ vertices.
    It remains to prove that every labeling $w \colon E(K_{pk+s})\to\Z_p$ contains a zero-sum copy of $W_{pk}(s)$ whenever $k\ge K_0(p,s)$.

    Suppose, on the contrary, that such a labeling $w$ contains no zero-sum $W_{pk}(s)$.
    We first explain why the congruence condition is useful.
    If $w(uv)=\lambda_u+\lambda_v+c$ holds for the labeling on the vertices of $K_{pk+s}$, then for every copy $W$ of $W_{pk}(s)$,
    \begin{equation}\label{eq:zero-vanishes}
        \sum_{uv\in E(W)}w(uv)=\sum_{v\in V(W)}d_W(v)\lambda_v+c|E(W)|=0.
    \end{equation}
    Indeed, the degrees in $W$ are $s+2$ for rim vertices and $pk$ for cone vertices, while $|E(W)|=(s+1)pk$; all three quantities are $0$ in $\Z_p$ by~\eqref{eq:congruence-condition}.

    For four distinct vertices $a,b,c,d$, we define the switch defect
    \begin{equation}\label{eq:switch-defect}
        \Delta(a,b;c,d)=w(ac)+w(bd)-w(ab)-w(cd).
    \end{equation}
    We can view this as the change in the weight of a cycle segment when the path $abcd$ is replaced by $acbd$, since the middle edge $bc$ is used in both paths.

    We claim that there do not exist $p-1$ pairwise vertex-disjoint ordered quadruples $(a_i,b_i,c_i,d_i)$ with $\Delta(a_i,b_i;c_i,d_i)\ne0$.
    Suppose to the contrary that such quadruples exist, and put
    $\Delta_i=\Delta(a_i,b_i;c_i,d_i)\ne0$ for $i=1,\ldots,p-1$.
    Let $A_i=\{0,\Delta_i\}$ for $i=1,\ldots,p-1$.
    Clearly, $|A_i|=2.$
    By the generalized Cauchy-Davenport theorem, we have
    $A_1+\cdots+A_{p-1}=\Z_p.$
    Since $k\ge4$, we have $pk\ge4(p-1)$.  We may therefore choose the $s$ cone vertices outside these quadruples and build a rim cycle of length $pk$ containing each path $a_ib_ic_id_i$ as a consecutive segment.
    Independently replacing any chosen segment $a_ib_ic_id_i$ with $a_ic_ib_id_i$ produces a copy of $W_{pk}(s)$ whose weight belongs to the set $S+A_1+\cdots+A_{p-1}$, where $S$ is the weight of the original copy.
    Since $A_1+\cdots+A_{p-1}=\Z_p,$ there must exist a $W_{pk}(s)$ with weight $0$, which is a contradiction.

    Choose a maximal family of pairwise vertex-disjoint ordered quadruples with non-zero switch defect, and let $F$ be the set of all vertices occurring in this family.
    The preceding paragraph implies that
    \begin{equation}\label{eq:D-bound}
        |F|\le F_0=4(p-2).
    \end{equation}
    Let $U=V(K_{pk+s})\setminus F$.
    By the maximality of $F$, every switch defect involving only vertices in $U$ must be zero.
    Thus,~\eqref{eq:four-point-relation} holds on $U$.
    Since $k\ge K_0(p,s)$, the set $U$ contains at least four vertices, and Lemma~\ref{lem:transfer weight} yields the representation $w(uv)=\lambda_u+\lambda_v+c$ for all $u,v\in U$.

    We extend the values $\lambda_u$ arbitrarily to vertices in $F$ and subtract the labeling $a(uv)=\lambda_u+\lambda_v+c$ from $w$. By~\eqref{eq:zero-vanishes}, this operation preserves the weight of every copy of $W_{pk}(s)$.
    Therefore, we may assume that
    \begin{equation}\label{eq:U-zero}
        w(uv)=0\quad\text{for all distinct }u,v\in U.
    \end{equation}

    Let $f=|F|$.
    For each $u\in U$, define its type with respect to $F$ by $\tau(u)=(w(xu))_{x\in F}\in\Z_p^f$.
    There are at most $p^f$ possible types. From the assumption that~$K_0(p,s)\ge \left\lceil\frac{(s+2F_0)p^{F_0}+F_0-s}{p}\right\rceil$ and~\eqref{eq:D-bound}, we obtain that
    \[
        |U|\ge pk+s-F_0\ge (s+2F_0)p^{F_0}\ge (s+2f)p^f.
    \]
    Hence, there is some type class $T\subseteq U$ satisfying $|T|\ge s+2f$.

    Let $H\subseteq T$ be a set of size $s$, which will form the set of cone vertices.
    Because $|T\setminus H|\ge2f,$ we can assign to each $x\in F$ a pair of vertices $\ell_x,r_x\in T\setminus H$ such that all $2f$ vertices are distinct.
    We now construct a Hamilton cycle on $V(K_{pk+s})\setminus H$ where each $x\in F$ appears as part of the segment $\ell_xxr_x$.
    Since the host graph is complete, we can easily achieve this by connecting these segments and arbitrarily inserting any remaining vertices of $U\setminus H$ between them.
    Let $C$ be the resulting rim cycle.

    We claim that the configuration $C+\overline K_s$ with cone set $H$ has a total weight zero.
    Indeed, by~\eqref{eq:U-zero}, any edge with both endpoints in $U$ has weight $0$, and by definition, there are no edges among cone vertices.
    Now, fix a vertex $x\in F$.
    Since all vertices of $T$ share the same type, there exists a constant $c_x\in\Z_p$ such that $w(xu)=c_x$ for every $u\in T$.
    The total weight contribution by the edges incident to $x$ is $2c_x$ from the two rim edges $x\ell_x$ and $xr_x$, and $sc_x$ from the edges joining $x$ to the cone set $H$.
    According to~\eqref{eq:congruence-condition}, we have
    \[
        2c_x+sc_x=(s+2)c_x=0.
    \]
    Summing over all $x\in F$ results in a total weight of $0$.
    This contradiction proves the desired upper bound.
    Combined with the vertex-count lower bound, we conclude that
    $R(W_{pk}(s),\Z_p)=pk+s$ for all $k\ge K_0(p,s)$, completing the proof of Theorem~\ref{thm:prime-case}.
\end{proof}

\section{The Cycle Case Modulo Three: Proof of Theorem~\ref{thm:q3-cycles}}

Before proving Theorem~\ref{thm:q3-cycles}, we first note that the value of $R(C_{3},\Z_3)=11$ was determined by Caro in~\cite{Caro1992Complete}.
We now prove Theorem~\ref{thm:q3-cycles}.
By Theorem~\ref{thm:reduction}, for every $k\ge 2$,
\[
    R(C_{3k},\Z_3)\le \max\{R(C_{6},\Z_3),3k+2\}.
\]
By Theorem~\ref{thm:lower}, for every $k\ge 1$,
$R(C_{3k},\Z_3)\ge 3k+2$.
Thus, to prove Theorem~\ref{thm:q3-cycles}, it suffices to establish the following lemma for $R(C_6,\Z_3).$

\begin{lemma}\label{lem:z3-c6-base}
    $R(C_6,\Z_3)=8$.
\end{lemma}

We will use the following lemma to express edge weights in terms of vertex weights.

\begin{lemma}\label{lem:z3-affine}
    Let $w\colon E(K_N)\to \Z_3$ be an edge-labeling.  Suppose that for every ordered pair of distinct vertices $(x,y)$, the difference
    $w(xv)-w(yv)$
    is independent of $v\in V(K_N)\setminus\{x,y\}$.
    Then there exist $\lambda_v$ for every $v\in V$ and a constant $c\in\Z_3$ such that for every edge $uv$,
    \[
        w(uv)=\lambda_u+\lambda_v+c.
    \]
\end{lemma}

\begin{proof}
    We fix a vertex $x$.
    For every $v\ne x$, consider the ordered pair $(v,x)$, let $\mu_v\coloneqq w(vt)-w(xt)$ be the common difference for every $t\in V(K_N)\setminus\{x,v\}$.
    Then, for distinct $u,v\ne x$ we have
    \[
        w(uv)=w(xu)+\mu_v=w(xv)+\mu_u.
    \]
    Thus, $w(xu)-\mu_u$ is independent of $u\ne x$; denote this common value by $c$.
    Let $\lambda_x=0$ and $\lambda_u=\mu_u$ for $u\ne x$.
    Then, we have $w(xu)=\mu_u+c=\lambda_x+\lambda_u+c$  for every $u\ne x$, and
    \[
        w(uv)=w(xu)+\mu_v=\lambda_u+\lambda_v+c
    \]
    for distinct $u,v\ne x$.
    This completes the proof of Lemma~\ref{lem:z3-affine}.
\end{proof}

\subsection{Proof of Lemma~\ref{lem:z3-c6-base}}
\begin{proof}
    Now, we prove $R(C_6,\Z_3)=8$.
    The lower bound $R(C_6,\Z_3)\ge 8$ follows from the construction in Theorem~\ref{thm:lower}.
    It remains to prove the upper bound $R(C_6,\Z_3)\le 8$.

    We first characterize the structure of edge-labelings of $K_7$ containing no zero-sum $C_6$.

    Let $w\colon E(K_7)\to\Z_3$ be an edge-labeling containing no zero-sum $C_6$.
    Fix distinct vertices $x,y$ and for $t\in V(K_7)\setminus\{x,y\}$, let
    \[
        r_t=w(xt)-w(yt).
    \]
    We first claim that among these five values, at least four are equal.
    Indeed, suppose on the contrary that no value occurs four times.
    Consider the largest value class,
    it has at least two vertices, and its complement also has at least two vertices.  Choosing two vertices $a,b$ from this class and two vertices $c,d$ from its complement, the two disjoint pairs $\{a,c\}$ and $\{b,d\}$ both have unequal $r$-values.
    Suppose that the $6$-cycle $xacybdx$ has weight $S$.
    Consider the other three $6$-cycles $xcaybdx$, $xacydbx$, and $xcaydbx$.
    It is not hard to see that they have weights
    \[
        \quad S+\alpha,
        \quad S+\beta,
        \quad S+\alpha+\beta
    \]
    with $\alpha=r_c-r_a\ne 0, \text{ and } \beta=r_b-r_d\ne0$.
    In $\Z_3$, it is impossible for all four of $S,S+\alpha, S+\beta, S+\alpha+\beta$ to be non-zero, contradicting the assumption that there is no zero-sum $C_6$.

    We call the pair $xy$ \emph{regular} if all five values $r_t$ are equal, and \emph{singular} otherwise.
    Thus, a singular pair has exactly one exceptional vertex.

    First, suppose that every pair is regular.
    Then, for every ordered pair $(x,y)$, the difference $w(xv)-w(yv)$
    is independent of $v\in V(K_7)\setminus\{x,y\}.$
    By Lemma~\ref{lem:z3-affine}, there exist $\lambda_v$ for every $v\in V(K_7)$ and a constant $c\in\Z_3$, such that for every $uv\in E(K_7),$ $w(uv)=\lambda_u+\lambda_v+c.$
    Let $T=\sum_v\lambda_v$.
    Note that a $6$-cycle on the vertex set $V(K_7)\setminus\{v\}$ has weight $2(T-\lambda_v)+6c=\lambda_v-T$ in $\Z_3$.
    Since no such cycle is zero-sum, $\lambda_v\ne T$ for every $v$.
    Hence, the $7$ labels $\lambda_v$ use only the two residues different from $T$.
    If one of these residues occurs $r$ times, then the identity $\sum_v\lambda_v=T$ gives $r\equiv2\pmod3$, so the two multiplicities are $5$ and $2$.  Consequently, there is a $5$-set $B$, two vertices $p,q$, and elements $b\in\Z_3$, $\delta\in\Z_3^\times$ such that
    \begin{equation}\label{Equ:regular structure for K7}
        w(BB)=b,
        \qquad
        w(B,\{p,q\})=b+\delta,
        \qquad
        w(pq)=b-\delta.
    \end{equation}
    Here, for example, $w(BB)=b$ means that every edge with both endpoints in $B$ has weight $b$.
    \medskip

    Now, suppose that some pair $xy$ is singular.
    We will obtain a structure similar to the one above.
    Let $z$ be its unique exceptional vertex and let
    $A=V(K_7)\setminus\{x,y,z\}.$
    Then $|A|=4$, and there are $a,b\in\Z_3$ with $a\ne b$ such that
    \begin{equation}\label{Equ:define a,b}
        w(xu)-w(yu)=a\quad(\text{for all } u\in A),
        \qquad
        w(xz)-w(yz)=b.
    \end{equation}
    Let $\delta=a-b\ne0$ and let $\ell_u=w(yu)$ for all $u\in A$.

    For distinct $u,v,s\in A$, consider the cycle $xzuyvsx$.
    Interchanging $z$ and $u$ changes its weight by $\delta$.
    Since both resulting $6$-cycles have a non-zero sum, the original cycle has weight $\delta$.
    Therefore, we get that
    \[
        w(xz)+w(zu)+w(uy)+w(yv)+w(vs)+w(sx)=\delta.
    \]
    Let
    \begin{equation}\label{Equ:Fu,Gvs}
        F_u=w(zu)+\ell_u, \quad G_{vs}=w(vs)+\ell_v+\ell_s.
    \end{equation}
    Since $w(sx)=a+\ell_s$, we obtain that
    \[
        w(xz)+F_u+G_{vs}+a=\delta.
    \]
    Thus, there is a constant $K\in \Z_3$ independent of $u,v,s$ such that for distinct $u,v,s\in A$,
    \[
        F_u+G_{vs}=K.
    \]
    For any two vertices $u,t\in A$, letting $\{v,s\}=A\setminus\{u,t\}$, the two triples $(u,v,s)$ and $(t,v,s)$ yield
    \[
        F_u+G_{vs}=F_t+G_{vs}=K.
    \]
    Thus, $F_u=F_t$, which implies that all $F_u$ are equal.
    Moreover, the equation $F_u+G_{vs}=K$ forces all $G_{vs}$ to be equal as well.

    For any two distinct vertices $u,v\in A$, let $d=\ell_u-\ell_v$. Since all $G_{vs}$ are equal and all $F_u$ are equal, combining~\eqref{Equ:define a,b} with~\eqref{Equ:Fu,Gvs}, the five differences $w(ut)-w(vt)$, where $t\notin\{u,v\}$, are
    \[
        d,d,-d,-d,-d.
    \]
    The two entries $d$ come from $t=x,y$; the three entries $-d$ come from $t=z$ and the other two vertices of $A$.
    Since at least four of the five differences must be equal, $d=0$.
    Thus, for $u\in A$, all $\ell_u$ are equal.
    From~\eqref{Equ:define a,b} and~\eqref{Equ:Fu,Gvs}, it follows that all edges inside $A$ have one common weight, and each of $x,y,z$ has a constant weight to $A$.

    By subtracting this common weight on $AA$ from every edge, this does not change any $C_6$-weight, since $6\equiv0\pmod3$.
    Thus, we may assume that $w(AA)=0$.
    Let
    \[
        X=w(x,A),
        \qquad
        Y=w(y,A),
        \qquad
        Z=w(z,A).
    \]
    Consider the cycles using two of $x,y,z$ and all four vertices of $A$ yield
    \begin{equation}\label{Equ:X,Y,Z}
        X+Y\ne0,
        \qquad
        X+Z\ne0,
        \qquad
        Y+Z\ne0.
    \end{equation}
    A cycle of the form $xuyvzsx$, with distinct $u,v,s\in A$ gives
    \begin{equation}\label{Equ:X+Y+Z}
        X+Y+Z\ne0.
    \end{equation}
    These four inequalities in~\eqref{Equ:X,Y,Z} and~\eqref{Equ:X+Y+Z} in $\Z_3$ imply that exactly one of $X,Y,Z$ is zero and the other two are equal and non-zero.
    Indeed, if none of $X,Y,Z$ is zero, then the pairwise inequalities force $X=Y=Z$, contradicting $X+Y+Z\ne0$.  If two of them are zero, then one pair sum is zero.  Hence exactly one is zero, and the other two must be equal and non-zero.
    Relabel $x,y,z$ so that $X=0$ and $Y=Z=\eta\ne0$.
    Let $h=w(xy)$, $p'=w(xz)$ and $m=w(yz)$.
    Cycles in which $xy$ is adjacent and the remaining four vertices lie in $A$, together with cycles using $xy$, $z$, and three vertices of $A$, imply that $h+\eta\ne0$ and $h\ne0$.
    Hence $h=\eta$.
    Similarly $p'=\eta$.
    The cycles in which $yz$ is adjacent and the remaining four vertices lie in $A$ give $2\eta+m\ne0$, so $m\ne\eta$.  If $m=2\eta$, then every pair in this $K_7$ is regular.  Indeed, pairs inside $A$ are regular; pairs between $A$ and one of $x,y,z$ are regular; and the three pairs among $x,y,z$ are regular by the identities $X=0$, $Y=Z=\eta$, $h=p'=\eta$, and $m=2\eta$.  This contradicts the assumption that we are in the singular case.  Hence $m\ne2\eta$.
    Therefore $m=0$.
    Let $B=A\cup\{x\}$.
    Thus, the singular case also has the same structural form:
    \begin{equation}\label{Equ:irregular structure for K7}
        w(BB)=0,
        \qquad
        w(B,\{y,z\})=\eta,
        \qquad
        w(yz)=0.
    \end{equation}

    Combining~\eqref{Equ:irregular structure for K7} with the regular case~\eqref{Equ:regular structure for K7}, we get that every $K_7$ containing no zero-sum $C_6$ has a $5$-set $B$, two vertices $p,q$, and elements $b\in\Z_3$, $\delta\in\Z_3^\times$ such that
    \begin{equation}\label{Equ:structure for K7}
        w(BB)=b,
        \qquad
        w(B,\{p,q\})=b+\delta,
        \qquad
        w(pq)\in\{b,b-\delta\}.
    \end{equation}

    Now, suppose that $w\colon E(K_8)\to\Z_3$ containing no zero-sum $C_6$.
    Choose 7 vertices and apply the structure conclusion~\eqref{Equ:structure for K7} described above.
    After replacing every edge $w(e)$ with $\delta^{-1}(w(e)-b)$, which preserves the zero-sum property of any $C_6$, we may assume that among those $7$ vertices, there is a 5-vertex set $B$ and two additional vertices $p,q$ satisfying
    \[
        w(BB)=0,
        \qquad
        w(Bp)=w(Bq)=1,
        \qquad
        w(pq)\in\{0,2\}.
    \]
    Let $x$ be the eighth vertex, and let $a_u=w(xu)$ for $u\in B$.
    For distinct $i,j\in B$, a $6$-cycle passing through $x$ and all five vertices of $B$ has weight $a_i+a_j$.
    Hence, $a_i+a_j\ne0$.
    For distinct $i,j,u,v\in B$, the cycle $xiupvjx$ has weight $a_i+a_j+2$.
    Hence $a_i+a_j\ne1$.
    Therefore $a_i+a_j=2$ for all distinct $i,j\in B$, which implies that $a_i=1$ for all $i\in B$.
    Finally, for distinct $i,j,k\in B$, the cycle $xipjqkx$ has all six edge weights equal to $1$, making it a zero-sum cycle in $\Z_3$.  This contradiction proves that every $K_8$-labeling contains a zero-sum $C_6$, which completes the proof of Lemma~\ref{lem:z3-c6-base}.
\end{proof}

\section{The Wheel Case Modulo Three: Proof of Theorem~\ref{thm:q3-wheel}}

We now prove Theorem~\ref{thm:q3-wheel}.  Recall that $W_m=C_m+K_1$ denotes the wheel with rim length $m$.

\begin{proof}[Proof of Theorem~\ref{thm:q3-wheel}]
    The lower bound follows from the vertex count: $W_{3k}$ has $3k+1$ vertices.
    It remains to prove that every labeling $w\colon E(K_{3k+1})\to\Z_3$ contains a zero-sum copy of $W_{3k}$.

    Let $n=3k$ and let $V$ be the vertex set of $K_{n+1}$.
    For a vertex $h$ and a Hamilton cycle $C$ on $V\setminus\{h\}$, we define
    \[
        \Phi_h(C)=\sum_{e\in C}w(e)+\sum_{v\in V\setminus\{h\}}w(hv).
    \]
    Thus, $\Phi_h(C)$ is the weight of the wheel with center $h$ and rim $C$.  Suppose, for the contrary, that
    \begin{equation}\label{Equ:Phi ne 0}
        \Phi_h(C)\ne0
    \end{equation}
    for every $h$ and every Hamilton cycle $C$ on $V\setminus\{h\}$.

    Fix distinct vertices $x,y$.
    For $v\in V\setminus\{x,y\}$, define $\rho(v)=w(xv)-w(yv)$.
    We first claim that among the $n-1$ values $\rho(v)$, at least $n-2$ must be equal.  Indeed, otherwise, a largest value class would contain at least two vertices, and its complement would also have at least two vertices.
    We could therefore choose two disjoint pairs $\{a,c\}$ and $\{b,d\}$ such that $\rho(a)\ne\rho(c)$ and $\rho(b)\ne\rho(d)$.
    Since $n\ge6$, there exists a vertex $h\notin\{x,y,a,b,c,d\}$.
    We partition the remaining vertices of $V\setminus\{h,x,y,a,b,c,d\}$ into two ordered paths $P$ and $Q$ (which may be empty), and form a Hamilton cycle on $V\setminus\{h\}$ of the form
    \[
        C=xaPcybQdx.
    \]
    Given a path $T=t_1,\ldots,t_k$, let $T^{\rm rev}$ denote the path $t_k,\ldots,t_1$.
    Interchanging $a,c$ means replacing the subpath $aPc$ by $cP^{\rm rev}a$; this changes $\Phi_h(C)$ by $\alpha\coloneqq\rho(c)-\rho(a)\ne0$.
    Similarly, interchanging $b,d$ means replacing $bQd$ by $dQ^{\rm rev}b$; this changes $\Phi_h(C)$ by $\beta\coloneqq\rho(b)-\rho(d)\ne0$.
    By interchanging the pairs $a,c$ or $b,d$, we will get four different wheels whose weights are $S,S+\alpha,S+\beta$ and $S+\alpha+\beta$ with $\alpha,\beta\ne0$.
    In $\Z_3$ these four elements cannot all be non-zero, which contradicts~\eqref{Equ:Phi ne 0}.
    Therefore, every pair $xy$ is either regular, meaning all values of $\rho(v)$ are equal, or singular, meaning there is exactly one exceptional vertex.
    If every pair is regular, Lemma~\ref{lem:z3-affine} yields
    \[
        w(uv)=\lambda_u+\lambda_v+c.
    \]
    However, in $W_n$, every rim vertex has degree $3$, the center has degree $n\equiv0\pmod3$, and $|E(W_n)|=2n\equiv0\pmod3$.
    Hence every copy of $W_n$ has total weight zero, which is a contradiction to~\eqref{Equ:Phi ne 0}.
    Therefore, there must exist a singular pair.

    Let $xy$ be a singular pair, let $z$ be its exceptional vertex, and put $A=V\setminus\{x,y,z\}$.
    Then $|A|=n-2\ge4$.
    By the definition of a singular pair, there exist $a,b\in\Z_3$ with $a\ne b$, such that
    \begin{equation}\label{Equ:singular wheel a,b}
        w(xu)-w(yu)=a\quad(\text{for all }u\in A),
        \qquad
        w(xz)-w(yz)=b.
    \end{equation}
    Let $\delta=a-b\ne0$ and $\ell_u=w(yu)$ for $u\in A$.  For $h\in A$, define
    \[
        D_h=\sum_{v\in V\setminus\{h\}}w(hv).
    \]

    Take distinct $u,h\in A$, and let $v_1,\ldots,v_s$ be an ordering of $A\setminus\{u,h\}$, where $s=n-4$.  Consider the wheel with center $h$ and rim
    \[
        xzuyv_1\cdots v_sx.
    \]
    Interchanging $z$ and $u$ changes the wheel's weight by $a-b=\delta$.
    Since both resulting wheels have non-zero sums, the original wheel must have weight $\delta$.
    Therefore, we obtain that
    \begin{equation}\label{Equ:wheel sum}
        w(xz)+w(zu)+w(uy)+w(yv_1)+\sum_{i=1}^{s-1}w(v_i v_{i+1})+w(v_sx)+D_h=\delta.
    \end{equation}
    We define
    \begin{equation}\label{Equ:define F and H}
        F(u)=w(zu)+\ell_u \quad \text{and}\quad
        H(v_1,\ldots,v_s)=\ell_{v_1}+\sum_{i=1}^{s-1}w(v_i v_{i+1})+\ell_{v_s}.
    \end{equation}
    Substituting the above equalities into~\eqref{Equ:wheel sum}, and by~\eqref{Equ:singular wheel a,b}, there is a constant $C_0\in \Z_3$ independent of $u,h,v_1,\ldots,v_s$ such that
    \begin{equation}\label{Equ:F,D,H}
        F(u)+D_h+H(v_1,\ldots,v_s)=C_0.
    \end{equation}
    Interchanging the roles of $u$ and $h$ in~\eqref{Equ:F,D,H} yields $D_h-F(h)=D_u-F(u)$.
    Thus, we may assume that there is a constant $c_1\in\Z_3$ such that $D_i-F(i)=c_1$ for every $i\in A$.
    Therefore, \eqref{Equ:F,D,H} implies that there is a constant $C_1\in\Z_3$ such that
    \begin{equation}\label{Equ:H+F+F=C}
        H(A\setminus\{i,j\})+F(i)+F(j)=C_1
    \end{equation}
    for every pair $i,j\in A$, and the value of $H$ is independent of the ordering of its arguments.

    First, assume that $k\ge3$.
    Then $|A|\ge7$.
    Since $|A|\ge7$, for any distinct $r,p,t\in A$ we may choose $i,j\in A\setminus\{r,p,t\}$.  Applying~\eqref{Equ:H+F+F=C} to the same pair $i,j$ and to two orderings of $A\setminus\{i,j\}$ beginning with $r,p,t,\ldots$ and $p,r,t,\ldots$, respectively, using~\eqref{Equ:define F and H}, we have
    \[
        \ell_r+w(rp)+w(pt)=\ell_p+w(pr)+w(rt),
    \]
    and hence, for a fixed $t$, the value $w(rt)-\ell_r$ is independent of $r\ne t$.
    By symmetry, $w(rt)-\ell_t$ is independent of $t\ne r$.
    Since $r$ and $t$ are arbitrarily chosen from $A$, there is a constant $c\in\Z_3$ such that $w(ij)=\ell_i+\ell_j+c$ with $i,j\in A$ and $i\ne j$.
    Substituting this into~\eqref{Equ:H+F+F=C} gives
    \begin{equation*}
        2\sum_{v\in A\setminus\{i,j\}}\ell_v+c(s-1)+w(zi)+\ell_i+w(zj)+\ell_j=C_1,
    \end{equation*}
    which implies that
    \begin{equation*}
        2\sum_{v\in A}\ell_v+c(s-1)+\bigl(w(zi)-\ell_i\bigr)+
        \bigl(w(zj)-\ell_j\bigr)=C_1.
    \end{equation*}
    Thus, there exists a constant $C_2\in \Z_3$, independent of the pair $i,j$, such that
    \[
        \bigl(w(zi)-\ell_i\bigr)+\bigl(w(zj)-\ell_j\bigr)=C_2
        \qquad(\text{for all }i\ne j\in A),
    \]
    and therefore, for some constant $g\in\Z_3$, we have
    \[
        w(zi)=\ell_i+g \qquad (\text{for all }i\in A).
    \]

    Now assume that $k=2$.
    Then, $|A|=4$.
    Equation~\eqref{Equ:H+F+F=C} states that, for every partition $A=\{r,s\}\cup\{u,h\}$, there exists a constant $C_3\in \Z_3$, independent of the partition, such that $w(rs)+w(zu)+w(zh)=C_3$.
    Let $m_u=w(zu)$, this gives $w(rs)=m_r+m_s+c'$ for some constant $c'$.
    For $u,v\in A$, compare the five differences from $u$ and $v$ to the other vertices.
    By~\eqref{Equ:singular wheel a,b}, the two differences to $x,y$ are $\ell_u-\ell_v$, while the three differences to $z$ and to the other two vertices of $A$ are $m_u-m_v$.
    Since at least four of the five differences are equal, $m_u-m_v=\ell_u-\ell_v$.
    Hence, for suitable constants $c,g\in\Z_3$, for all $u\in A$, $m_u=\ell_u+g$ and again
    \[
        w(uv)=\ell_u+\ell_v+c,
        \qquad
        w(zu)=\ell_u+g
        \qquad(\text{for all } u,v\in A,\ u\ne v)
    \]

    In all cases, for $u,v\in A$, $u\ne v$, we have
    \begin{equation}\label{Equ:wheel structure}
        w(yu)=\ell_u,
        \qquad
        w(xu)=\ell_u+a,
        \qquad
        w(zu)=\ell_u+g,
        \qquad
        w(uv)=\ell_u+\ell_v+c.
    \end{equation}
    Define a weight $q_0(uv)=\lambda_u+\lambda_v+c$ by
    \[
        \lambda_u=\ell_u\ (\text{for all }u\in A),
        \qquad
        \lambda_y=-c,
        \qquad
        \lambda_x=a-c,
        \qquad
        \lambda_z=g-c.
    \]
    Then, $q_0$ agrees with $w$ on every edge incident with at least one vertex of $A$.  Replacing $w$ with $w-q_0$ preserves the weight of every $W_n$, because every vertex degree in $W_n$ is divisible by $3$ and $|E(W_n)|=2n\equiv0\pmod3$.
    Hence, we may assume that all edges incident with $A$ have weight $0$.
    Only the three edges $xy,xz,yz$ may be non-zero.

    We choose a center $h\in A$.
    Since $|A\setminus\{h\}|=n-3\ge3$, we can order the rim vertices so that $x,y,z$ are pairwise non-adjacent.
    For instance, we may choose $xa_1ya_2za_3\cdots x$
    with $a_1,a_2,a_3\in A\setminus\{h\}$ distinct and all remaining vertices of $A\setminus\{h,a_1,a_2,a_3\}$ inserted in the final path.
    This wheel uses no edge among $x,y,z$, and all its other edges are incident with $A$.
    Hence, its total weight is zero, contradicting the assumption~\eqref{Equ:Phi ne 0}.

    Therefore, every labeling of $K_{3k+1}$ contains a zero-sum $W_{3k}$, and
    \[
        R(W_{3k},\Z_3)=3k+1,
    \]
    which completes the proof of Theorem~\ref{thm:q3-wheel}.
\end{proof}

\section*{Acknowledgement}
The authors are grateful to Alexandru Malekshahian for his helpful comments that improved the presentation, as well as for sharing his insights on this problem and providing relevant references.

\section*{Appendix A: Justification of~$(35q^2)^{1/q}<F_q$ for $q\ge 3 $}\label{app:ineq-proof}
Recall that $F_q=\frac32+\frac{35}{2(q-1)}-\frac{1}{2q(q-1)}$.
For $3\le q\le 6$, the following elementary bounds suffice:
\[
    \begin{array}{c|c|c}
        q & (35q^2)^{1/q}\text{ is less than} & F_q\text{ is greater than} \\ \hline
        3 & 7                                 & 10                         \\
        4 & 5                                 & 7                          \\
        5 & 4                                 & 5                          \\
        6 & 4                                 & 4
    \end{array}
\]
For example, the left inequalities follow from
$315<7^3$, $560<5^4$, $875<4^5$, and $1260<4^6$.

For $7\le q\le 12$, we have $35q^2<3^q$ since it holds at $q=7$ and the ratio $3^q/q^2$ is increasing for $q\ge 2$.
Thus, $(35q^2)^{1/q}<3$.
Additionally, we have $F_q>3$
for $7\le q\le 12$.
Hence, the desired inequality holds in this range.

For $13\le q\le 35$, we have $35q^2<2^q$, since it holds at $q=13$ and the ratio $2^q/q^2$ is increasing for $q\ge 4$.
Thus, $(35q^2)^{1/q}<2$.
Moreover $F_q>2$ for $q\le 35$, since $F_q>2$ is equivalent to $36-\frac1q>q$, which is true for $13\le q\le 35$.

Finally, for $q\ge 36$ we have
\[
    35q^2<\left(\frac32\right)^q.
\]
Indeed, this holds at $q=36$, since
\[
    \left(\frac32\right)^{36}=\left(\frac94\right)^{18}>2^{18}>35\cdot 36^2,
\]
and the ratio $(3/2)^q/q^2$ is increasing for $q\ge 5$.
Thus, $(35q^2)^{1/q}<3/2$.  Since $F_q>3/2$ for every $q\ge 3$, the desired inequality follows.
\qed

\end{document}